\documentclass[10pt]{amsart}
\usepackage{amscd,graphicx,epsfig}
\usepackage[colorlinks=true, pdfstartview=FitV, linkcolor=blue, citecolor=blue, urlcolor=blue]{hyperref}
\usepackage{amsmath}
\usepackage{amssymb}
\usepackage{comment}
\usepackage{amsmath}
\usepackage{mathtools}
\usepackage{hyperref}
\usepackage{subcaption}
\usepackage{graphicx}
\usepackage{tikz}
\usepackage{tikz-cd}
\usetikzlibrary{matrix,arrows,positioning,calc,chains}
\usepackage{enumitem}

\title[Characterization of affine $\mathbb{G}_m$-surfaces of hyperbolic type]{Characterization  of  affine $\mathbb{G}_m$-surfaces of hyperbolic type}
\author{Andriy Regeta}

\address{\noindent Institut f\"{u}r Mathematik, Friedrich-Schiller-Universit\"{a}t Jena, \newline
\indent  Jena 07737, Germany}
\email{andriyregeta@gmail.com}

\newtheorem{claim}{Claim}

\newtheorem{theorem}{Theorem}

\newtheorem{lemma}{Lemma}
\newtheorem{proposition}{Proposition}
\newtheorem{definition and proposition}{Definition and proposition}

\theoremstyle{definition}
\newtheorem{definition}{Definition}
\newtheorem{remark}{Remark}
\newtheorem{example}{Example}

\newcommand{\name}[1]{\textsc{#1\/}}

\DeclareMathOperator{\Ve}{Vec}

\DeclareMathOperator{\id}{id}

\DeclareMathOperator{\Aut}{Aut}

\DeclareMathOperator{\LND}{LNV}

\def\Dp{\mathrm{D}_p}

\def\Dz(z-1){\mathrm{D}_{z(z-1)}}
\def\Dq{\mathrm{D}_q}

\newcommand{\lielnd}[1]{\langle \LND(#1) \rangle}

\def\O{{\mathcal O}}

\def \itt #1,#2:{\medskip\item[$\bullet$] 
     page\ \ignorespaces#1, line\ \ignorespaces#2:\ \ignorespaces}

\begin{document}

\begin{abstract}
In this note we extend the result from \cite{LRU19} and prove that 
  if $S$ is an affine non-toric $\mathbb{G}_m$-surface of hyperbolic type that admits a $\mathbb{G}_a$-action and $X$ is an affine irreducible variety
  such that $\Aut(X)$ is isomorphic to $\Aut(S)$ as an abstract group, then $X$ is a $\mathbb{G}_m$-surface of hyperbolic type.
Further, we show
 that a smooth
Danielewski surface $\Dp = \{ xy = p(z) \} \subset \mathbb{A}^3$, where $p$ has no multiple roots, is determined by its automorphism group seen as an ind-group in the category of affine irreducible varieties. 

 
%
\end{abstract}

\maketitle



\section{Introduction and main results}


In this note we work over the field of complex numbers $\mathbb{C}$ and assume all varieties to be irreducible. We denote
the multiplicative and additive group of $\mathbb{C}$ respectively by $\mathbb{G}_m$ and $\mathbb{G}_a$, and call a variety that admits
a regular $\mathbb{G}_m$-action/$\mathbb{G}_a$-action to be a $\mathbb{G}_m$-variety/$\mathbb{G}_a$-variety.\par
Affine $\mathbb{G}_m$-surfaces appear in three types  with respect to the dynamical behavior of the $\mathbb{G}_m$-action:  \emph{parabolic}, corresponding to the case with  infinitely many fixed points,
\emph{elliptic},  where an attractive fixed point exists, and \emph{hyperbolic}, with finitely many non-attractive 
fixed points.
In addition, in the case of a non-toric affine $\mathbb{G}_m$-surface $S$, the dynamical type of $S$ does not depend on the choice of the
$\mathbb{G}_m$-action  (\cite[Corollary 4.3]{FZ05}).
\cite[Theorem 1.2]{LRU18} shows that if there exists a group isomorphism of
automorphism groups of
normal affine  $\mathbb{G}_m$-surfaces $S$ and $S'$ with $S$ being also a non-toric $\mathbb{G}_a$-surface, then $S$ and $S'$
have the same dynamical type.
We generalize  this statement for the hyperbolic dynamical type, which can be viewed as a rigidity result,
and observe that it does not extend to parabolic and elliptic dynamical types in Remark \ref{remarkellipticandparabolic}.

\begin{theorem}\label{main2}
 Let $S$ and $X$ be   affine  $\mathbb{G}_m$-varieties  with $S$ being also a normal non-toric  $\mathbb{G}_a$-surface.  If there exists a group isomorphism 
$\varphi\colon \Aut(S) \to \Aut(X)$, 
  then $\dim X =2$. Moreover, if $X$  is normal, then $X$ is also a hyperbolic non-toric $\mathbb{G}_m$-surface.
\end{theorem}

 \begin{remark}\label{remarkellipticandparabolic}
 The only elliptic $\mathbb{G}_m$-surfaces $S$ that admit a $\mathbb{G}_a$-action are toric, thus Theorem \ref{main2}  does not extend to this case. Indeed,  by \cite[Lemma 1.10]{Lie10} $S$  admits a root subgroup (see Section \ref{rootsubgroups}) in its automorphism group with respect to a subtorus of $\Aut(S)$ induced by the $\mathbb{G}_m$-action.
Hence, by \cite[Theorem 3.3]{FZ05b} $S$ is
toric.
Moreover, at the end of this note, Example \ref{exampleparabolic}  demonstrates 
that Theorem \ref{main2} does not extend to the parabolic case.
\end{remark}

One of the most important class of hyperbolic $\mathbb{G}_m$-surfaces are Danielewski surfaces which were originally studied in \cite{Dan89}  with the aim
to find a counterexample for the generalized version of Zariski Cancellation Problem. 
A smooth Danielewski surface $\Dp$ is $\{  (x,y,z) \in \mathbb{A}^3 \mid xy = p(z) \}$, where $p$ is a polynomial without multiple roots.

\begin{remark}\label{Daigle}
%
 \cite[Lemma 2.10]{Da04}  shows that  $\Dp$ and $\Dq$ are isomorphic if and only if there exists an automorphism  $F$ of $\mathbb{C}[z]$ such that $\frac{F(p)}{q} \in \mathbb{C^*}$.  In particular if $\deg p = \deg q =2$, then $\Dp$ and $\Dq$ are isomorphic.
\end{remark}

We call a surface $\Dp$ \emph{generic} if there is no affine 
 automorphism of the affine line $\mathbb{C}$ that permutes the roots of $p$.
For two generic surfaces $\Dp$ and $\Dq$  with $\deg p \ge 3$ and $\deg q \ge 3$ there is an isomorphism 
$\Aut(\Dp)  \xrightarrow{\sim} \Aut(\Dq)$ of abstract groups.
Indeed,  in
\cite[Theorem and Remark (3) on p. 256]{ML90} and more precisely in \cite[Theorem 2.7]{KL16}
 it is shown that for a generic Danielewski surface
$\Dp$, we have 
$\Aut(\Dp) \simeq (\mathbb{C}[x] \ast \mathbb{C}[y]) \rtimes (\mathbb{G}_m
 \rtimes \mathbb{Z}/2\mathbb{Z})$ and the semidirect product
structure does not depend on $p(z)$ (see also \cite[Remark 7]{LR20}). 
On the other hand, by \cite[Theorem 3]{LR20} $\Aut(\Dp)$ and $\Aut(\Dq)$ are isomorphic as ind-groups (see Section \ref{Ind-groups} for details) if and only if $\Dp$ is isomorphic to $\Dq$ as a variety. 
In this paper we  prove  the following
result.

 \begin{theorem}\label{main1}
 Let $X$ be an affine irreducible variety. Assume $\Aut(X)$ is isomorphic to $\Aut(\Dp)$ as an ind-group, then $X$ is isomorphic to $\Dp$ as a variety.
 \end{theorem}
 
\subsection*{Acknowledgements}
I am thankful to Roland Pucek who has read the text and suggested a number of  improvements.


\section{Preliminaries}\label{Preliminaries}

    \subsection{Ind-groups}\label{Ind-groups}
The notion of an \emph{ind-group} goes back to \name{Shafarevich}, who called them infinite dimensional groups, see \cite{Sh66}. We refer to  \cite{FK15} and \cite{KPZ15}   for basic notions. 

\begin{definition}
 An \emph{ind-variety}  is  a set $V$ together with an ascending filtration $V_0 \subset V_1 \subset V_2 \subset ... \subset V$ such that the following is satisfied:

(1) $V = \bigcup_{k \in \mathbb{N}} V_k$;

(2) each $V_k$ has the structure of an algebraic variety;

(3) for all $k \in \mathbb{N}$ the subset $V_k \subset V_{k+1}$ is closed in the Zariski  topology.
\end{definition}

A \emph{morphism} between two ind-varieties $V = \bigcup_k V_k$ and $W = \bigcup_l W_l$ is a map $\phi\colon V \to W$ such that for each $k \in \mathbb{N}$ there is an $l \in \mathbb{N}$ such that 
$\phi(V_k) \subset W_l$ and such that the induced map $V_k \to W_l$ is a morphism of algebraic varieties. \emph{Isomorphisms} of ind-varieties are defined in the usual way.

An ind-variety $V = \bigcup_k V_k$ has a natural topology: a subset $S \subset V$ is called \emph{closed}, respectively \emph{open}, if $S_k := S \cap  V_k \subset V_k$ is closed, respectively open, for all $k$.  A closed subset $S \subset V$ has the natural structure of an ind-variety. It is called an \emph{ind-subvariety}. An ind-variety $V$ is called \emph{affine} if each variety $V_k$ is affine. In the sequel we consider only affine ind-varieties and  for simplicity we call them just ind-varieties.

A set theoretical product of ind-varieties admits a natural structure of an ind-variety.
 This allows us to introduce the following definition.

\begin{definition}
An ind-variety $G$ is called an \emph{ind-group} if the underlying set $G$ is a group such that the map $G \times G \to G$, defined by $(g,h) \mapsto gh^{-1}$, is a morphism of ind-varieties.
\end{definition}

Note that any closed subgroup $H$ of $G$, i.e. $H$ is a subgroup of $G$ and is a closed subset, is again an ind-group under the  closed ind-subvariety structure on $G$.
 A closed subgroup $H$ of an ind-group $G$ is  an \emph{algebraic subgroup} if $H$ is an algebraic subset of $G$
 i.e. $H$ is a closed subset of some $G_i$, where $G_1 \subset G_2 \subset ...$ is a filtration of $G$.

The next result  can be found in \cite[Section 5]{FK15}.
\begin{proposition}\label{ind-group}
Let X be an affine variety. Then $\Aut(X)$ has the structure of an ind-group such that for any algebraic group $G$, a regular $G$-action on X induces an ind-group homomorphism $G \to \Aut(X)$.
\end{proposition}

For example, if $X = \mathbb{A}^n$, the ind-group filtration of $\Aut(\mathbb{A}^n)$ is given by 
$$\Aut(\mathbb{A}^n)_d = \{ f = (f_1,\dots,f_n) \in \Aut(\mathbb{A}^n) \mid \deg f = \max_i f_i \le d, \; \deg f^{-1} \le d  \}.$$

The following observation will turn out to be useful in the proof of Theorem \ref{main2}:
\begin{lemma}[Lemma 2.4, \cite{LRU18}]\label{centralizerclosed}
Let X be an affine variety and let $U \subset \Aut(X)$ be a commutative subgroup that coincides with its centraliser. Then $U$ is
a closed subgroup of $\Aut(X)$.
\end{lemma}

 The next statement follows from \cite[Theorem B]{CRX19}  (see also \cite[Corollary 3.2]{RvS21}). We will need it in the proof of Theorem \ref{main2}.

\begin{proposition}\label{CRX}
Let $X$ be an affine variety and let $G$ be a commutative connected
closed subgroup of $\Aut(X)$. Then $G$ is the countable union of an increasing filtration
by commutative connected algebraic subgroups of $\Aut(X)$.
\end{proposition}

\subsection{Algebraic and divisible elements}
We call an element $f$ in a group $G$  \emph{divisible by $n$} if there exists an element $g \in G$ such
that $g^n = f$. An element is called \emph{divisible} if it is divisible by all $n \in \mathbb{Z}^+$.  An element $f$ in the automorphism group of an affine variety $X$ is \emph{algebraic} if it is contained in an algebraic subgroup $G$ of $\Aut(X)$
with respect to its ind-group structure.

In the proof of Theorem \ref{main2} we will need the following result that
follows from \cite[Theorem 3.1]{LRU18} and   \cite[Corollary 2.6]{LRU18}
and that
connects the notions of divisibility and algebraicity in the automorphism group of an affine surface.

\begin{proposition}\label{divisiblealgebraic}
Let $S$ be an affine irreducible algebraic surface.  Then the following two conditions are equivalent:
	\begin{enumerate}
			\item there exists a $k>0$ such that $f^k$ is divisible;
			\item $f$ is algebraic.
		\end{enumerate}
\end{proposition}

\subsection{Lie algebras of vector fields}\label{Liealgebras_vf}
We denote by $\Ve(X)$  the Lie algebra of  vector fields on an affine variety $X$.
A vector field $\nu \in \Ve(X)$ is called \emph{locally nilpotent} if for any $f \in \O(X)$ there exists $s \in \mathbb{N}$ such that $\nu^s(f) = 0$. 
By $\lielnd{X}$ we denote the Lie subalgebra of $\Ve(X)$ generated by all
locally nilpotent vector fields.
We have the following lemma.

\begin{lemma}[Lemma 1, \cite{LR20}]\label{iso-lnd}\label{vectorfields}
Let $X$ and $Y$ be affine algebraic varieties and let $\varphi\colon \Aut(X) \to \Aut(Y)$ be a homomorphism of ind-groups. Then $\varphi$ induces 
the homomorphisms of Lie algebras \[ d\varphi  \colon \lielnd{X} \to \lielnd{Y}. \] 
Moreover, if $\varphi$ is an isomorphism, then the homomorphism $d\varphi$  is also the isomorphism.
\end{lemma}


\subsection{Root subgroups}\label{rootsubgroups}
	In this section we describe  {\it root subgroups} of the automorphism group
	$\Aut(X)$ of an affine variety $X$ with respect to a subtorus.
	
	\begin{definition}
		Let $T \subset \Aut(X)$ be a subtorus in $\Aut(X)$, i.e. a closed
		algebraic subgroup isomorphic to a torus.  A closed algebraic subgroup
		$U \subset \Aut(X)$ isomorphic to $\mathbb{G}_a$ is called a {\it root
			subgroup} with respect to $T$ if the normalizer of $U$ in
		$\Aut(X)$ contains $T$.
		 Such an algebraic
		subgroup $U$ corresponds to a non-trivial  $\mathbb{G}_a$-action on $X$,
	  whose image in $\Aut(X)$ is normalized
		by $T$.
	\end{definition}
	
	Assume $U\subset \Aut(X)$ is a root subgroup with respect to $T$. Since
	$T$ normalizes $U$, we can define an action
	$\varphi\colon T\rightarrow \Aut(U)$ of $T$ on $U$ given by
	$t.u=t\circ u\circ t^{-1}$ for all $t\in T$ and $u\in U$.
	Moreover, since $\Aut(U)\simeq \mathbb{G}_m$, such an action corresponds to a
	character of the torus $\mu\colon T\rightarrow \mathbb{G}_m$, which does not
	depend on the choice of an isomorphism between $\mathbb{G}_m$  and
	$\Aut(U)$. Such a character is
	called the \emph{weight} of $U$. 
	The subgroup of $\Aut(X)$  generated by 
    algebraic subgroups $T$ and $U$
    is isomorphic to $\mathbb{G}_a\rtimes_\mu T$.

		Assume that the algebraic torus $T$ acts linearly and rationally  on a
	vector space $A$ of countable dimension. We say that $A$ is {\it
		multiplicity-free} if the weight spaces $A_{\mu}$ are all of
	dimension less or equal than one for every character
	$\mu\colon T\rightarrow \mathbb{G}_m$ of the torus $T$. In the proof of
	Theorem \ref{main2} we will use the following lemma that is due to
	Kraft:
	
	\begin{lemma}[Lemma~5.2, \cite{Kr17}]\label{Kr15local}
		Let $X$ be a normal affine variety and let $T \subset \Aut(X)$ be
		a torus. If there exists a root subgroup $U \subset \Aut(X)$ with
		respect to $T$ such that $\mathcal{O}(X)^{U}$ is multiplicity-free,
		then $\dim T \le \dim X \le \dim T +1$.
	\end{lemma}

	Let $X$ be an affine variety and consider a nontrivial algebraic action of $\mathbb{G}_a$ on $X$, given by $\lambda \colon \mathbb{G}_a \to \Aut(X)$. If $f \in \mathcal{O}(X)$ is a $\mathbb{G}_a$-invariant regular function, then the \emph{modification} $f \cdot \lambda$ of $\lambda$ is defined in the    following way (see \cite[Section 8.3]{FK15}):
$$(f \cdot \lambda)(s)x =\lambda(f(x)s)x$$ 
for $s \in \mathbb{C}$ and $x \in X$.
This is again a $\mathbb{G}_a$-action. 
 If $X$ is irreducible and $f  \neq 0$, then $f \cdot \lambda$  and $\lambda$ have the same invariants.
If $U \subset \Aut(X)$ is a closed algebraic subgroup isomorphic to $\mathbb{G}_a$ and if $f \in \mathcal{O}(X)^U$ is a $U$-invariant, then  similarly as above  we define  the modification $f \cdot U$ of $U$. Pick an isomorphism
$\lambda \colon \mathbb{G}_a \to U$ and set 
$$f \cdot U = \{ (f \cdot \lambda)(s) \mid s \in \mathbb{G}_a \}.$$ 
 
If $U \subset \Aut(X)$ is a root subgroup with respect to $T$ and $f \in \mathcal{O}(X)^U$ is a $T$-semi-invariant, then $f \cdot U$ is again a root subgroup with respect to $T$.

\subsection{Hyperbolic surfaces}

We will use the next two lemmas proved in \cite{LRU18} in the proof of Theorem \ref{main2}.

\begin{lemma}[Lemma 4.16, \cite{LRU18}]\label{hyperbolic} A non-toric $\mathbb{G}_m$-surface $S$ admits root subgroups of different weights if and only if $S$ is hyperbolic. Furthermore, in this case all
root subgroups have different weights.
\end{lemma}

\begin{lemma}[Lemma 4.15, \cite{LRU18}]\label{LRULemma4.15}
Let $S$ be a non-toric normal affine $\mathbb{G}_m$-surface and denote by 
$T \subset \Aut(S)$ the subgroup
isomorphic to $\mathbb{G}_m$ induced by the $\mathbb{G}_m$-action. Let 
$H \subset \Aut(S)$ be an abelian subgroup containing only
algebraic elements such that $T \subset H$. Then there exists a finite group $F$ such that $H \simeq T \times F$.
\end{lemma}

\section{Proof of Theorem \ref{main2}}

Assume $\varphi\colon \Aut(S) \to \Aut(X)$ is an isomorphism of groups. Denote by $T$  a one-dimensional subtorus of $\Aut(S)$, i.e. $T \subset \Aut(S)$ is an algebraic subgroup isomorphic to $\mathbb{G}_m$.  Further,
denote by $Z$ a maximal abelian subgroup of 
the centralizer of $T$ in $\Aut(S)$ that contains $T$.
The group $Z$ coincides with its centralizer in $\Aut(S)$
which in turn implies that $Z \subset \Aut(S)$ is closed (see Lemma \ref{centralizerclosed}).
We claim that
$Z$ is a countable extension of $T$. Indeed, otherwise if 
$Z$  is an uncountable extension of $T$, then the connected component $Z^\circ$ is again an uncountable extension of $T$ (as $Z^\circ \subset Z$ is a countable index subgroup) and by
 Proposition \ref{CRX} 
$Z$ contains a two-dimensional commutative non-unipotent algebraic subgroup  which is not possible as $S$ is non-toric (see \cite[Lemma 4.17]{LRU18}).
Further, since $\varphi(Z)$ coincides with its centralizer in $\Aut(X)$, 
$\varphi(Z) \subset \Aut(X)$ is closed and $\varphi(Z)^\circ$ is a union of commutative algebraic groups (see Proposition \ref{CRX}). 

By assumption $S$ admits  a
regular $\mathbb{G}_a$-action. Hence,
$\Aut(S)$ contains a root subgroup with respect to $T$ (see for example
\cite[Lemma 1.10]{Lie10}).
 Choose a root subgroup  $U \subset \Aut(S)$ with respect to $T$. Note that the weight of the root subgroup $U$ is non-zero since $S$ is non-toric (see  \cite[Lemma 4.17]{LRU18}).
\begin{claim}\label{closedrootsubgroup}
 The subgroup $\varphi(U) \subset \Aut(X)$ is closed.
 \end{claim}
 
 The subgroup $U \subset \Aut(S)$ is normalized by $T$ and in particular, there is $t_0 \in T$ such that 
\[
t_0 \circ u \circ t_0^{-1} = u^2,
\]
where $u \in U$.
We claim that
\begin{equation}\label{rootsubgroupequalscharacter}
U = \{ u \in \Aut(S) \mid t_0 \circ u \circ t_0^{-1} = u^2 \}.
\end{equation}
Indeed, if there is some $h \in \Aut(S) \setminus U$ such that $t_0 \circ h \circ t_0^{-1} = h^2$, then the group generated by $t_0$ normalizes the group generated by $h$. Hence, $T = \overline{\langle t_0 \rangle}$ normalizes  $\overline{\langle h \rangle}^\circ$. Observe that $\overline{\langle h \rangle}^\circ$ is one-dimensional algebraic subgroup since otherwise  $\overline{\langle h \rangle}^\circ$ would contain a two-dimensional commutative non-unipotent algebraic subgroup 
which contradicts the assumption that $S$ is non-toric (see  \cite[Lemma 4.17]{LRU18}).
Hence, $\overline{\langle h \rangle}^\circ \simeq \mathbb{G}_a$ (see \cite[Lemma 4.10]{LRU18}). In another words, $\overline{\langle h \rangle}^\circ \subset \Aut(S)$ is the root subgroup with respect to $T$. Since all root subgroups of $\Aut(S)$ with respect to $T$ have different weights (see Lemma \ref{hyperbolic}) we have that $\overline{\langle h \rangle}^\circ \subset U$ which proves \eqref{rootsubgroupequalscharacter}. The Claim \ref{closedrootsubgroup} follows from \eqref{rootsubgroupequalscharacter}.

\begin{claim}\label{connectedimageroot}
$\varphi(U) = \varphi(U)^\circ$.
\end{claim}

	 As $T$ acts transitively on $U \setminus \{\id_S\}$ by conjugations, 
		it follows that
		$\varphi(T)$ acts transitively on $\varphi(U) \setminus \{\id_X\}$ by conjugations. Note that
		$\varphi(U)^\circ \setminus \{ \id_X \}$ is a subset of $\varphi(U) \setminus \{\id_X\}$ that is left
		invariant under the $\varphi(T)$-action. As $\varphi(U)$ is uncountable
		and $\varphi(U)^\circ$ has countable index in $\varphi(U)$, it follows
		that $\varphi(U)^\circ \setminus \{ \id_X \}$ is non-empty.
		Hence,  $\varphi(U)^\circ \setminus \{ \id_X \} = \varphi(U) \setminus \{\id_X\}$ and thus 
		$\varphi(U)^\circ = \varphi(U)$. This proves Claim~\ref{connectedimageroot}.

\begin{claim}\label{connectedt}
$\overline{\varphi(T)}^\circ$ is isomorphic to $\mathbb{G}_m$.
\end{claim}

The subgroup $\overline{\varphi(T)}^\circ \subset \Aut(X)$ is  closed. Moreover,  $\varphi(T)$ acts on $\varphi(U)$ by conjugations and hence, $\overline{\varphi(T)}^\circ$ acts on $\overline{\varphi(U)}^\circ = \varphi(U)$ by conjugations.
By Proposition \ref{CRX} the subgroups $\overline{\varphi(T)}^\circ,\varphi(U)\subset \Aut(X)$ are the unions of commutative connected algeraic groups $\bigcup_{i \ge 1} G_i$ and $\bigcup_{i \ge 1} H_i$ respectively. Moreover, since $\varphi(U)$ does not contain elements of finite order, $H_i \simeq \mathbb{G}_a^{r_i}$ for each $i \ge 1$.
Assume there is $k \in \mathbb{N}$ such that $G_k$ contains an algebraic subgroup $K \simeq \mathbb{G}_a$. Then   it follows from the Lie-Kolchin Theorem (see \cite[§17.6]{Hum75}) that there is $l \in \mathbb{N}$ and $v \in H_l$ such that $K$ fixes $v$. Equivalently, $\varphi^{-1}(K)$ fixes $\varphi^{-1}(v) = u \in U$. 

All elements of $\varphi^{-1}(K)$ are divisible which in turn implies that all elements of $\varphi^{-1}(K)$ are algebraic (see Proposition \ref{divisiblealgebraic}). Hence,
by Lemma 
\ref{LRULemma4.15} the
commutative subgroup $\varphi^{-1}(K) \subset \Aut(S)$
is a subgroup of $T \times F$ for some finite group $F$.
 Since all divisible elements of $T \times F$ are contained in $T$ we conclude that $\varphi^{-1}(K) \subset T$. Further, because 
 $\varphi^{-1}(K)$ fixes $\varphi^{-1}(v) = u \in U$ it follows that  $T = \overline{\varphi^{-1}(K)}$ acts trivially on  $\overline{\langle u  \rangle} = U$ which is not the case.
Therefore, $G_i \simeq \mathbb{G}_m^{r_i}$ for some $r_i \in \mathbb{N}$.
 Moreover, since all elements of $\mathbb{G}_m^{r_i}$ are divisible, similarly as above, $\varphi^{-1}(G_i) \subset T$. Hence, there is no copy of 
  $(\mathbb{Z}/p\mathbb{Z})^2$ in $\varphi^{-1}(G_i)$, where $p > 1$ and $i \in \mathbb{N}$. This implies that $G_i \simeq \mathbb{G}_m$ for any $i$ and we conclude that 
   $\overline{\varphi(T)}^\circ \simeq \mathbb{G}_m$.
 
 \begin{claim}\label{connectedtorus}
$\varphi(T)$ is isomorphic to $\mathbb{G}_m$ and $\varphi(U)$ is isomorphic to $\mathbb{G}_a$.
\end{claim}

We first claim that $\varphi(T) = \overline{\varphi(T)}^\circ$. Indeed,
since all elements of $\varphi^{-1}(\overline{\varphi(T)}^\circ)$ are divisible, by Proposition \ref{divisiblealgebraic} they are all algebraic. 
By Lemma \ref{LRULemma4.15} the group  $\varphi^{-1}(\overline{\varphi(T)}^\circ)$  is a subgroup of $T \times F$ for some finite group $F$. Moreover, since all divisible elements of $T \times F$ are contained in $T$ we conclude that $\varphi^{-1}(\overline{\varphi(T)}^\circ) \subset T$ or equivalently $\overline{\varphi(T)}^\circ \subset \varphi(T)$.
 Because both 
$\overline{\varphi(T)}^\circ$ and $\varphi(T)$ act on $\varphi(U)\setminus \{ \id_X \}$ transitively with  finite kernels, it follows that the subgroup 
$\overline{\varphi(T)}^\circ \subset \varphi(T)$ has a finite index. Finally, because $\overline{\varphi(T)}^\circ$ and $\varphi(T)$, seen as abstract groups, are isomorphic to $\mathbb{G}_m$, the torsion subgroups of $\overline{\varphi(T)}^\circ$ and $\varphi(T)$ coincide and we conclude that 
 $\overline{\varphi(T)}^\circ = \varphi(T) \simeq \mathbb{G}_m$. Since $\varphi(T)$ acts on $\varphi(U) \setminus \{ \id_X \}$ transitively  with a finite kernel, $\varphi(U)$ is a one-dimensional algebraic group which implies that $\varphi(U) \simeq \mathbb{G}_a$ (see \cite[Lemma 4.10]{LRU18}).

\begin{claim}\label{connectedtorus}
$\varphi(T)$ acts on the invariant ring $\mathcal{O}(X)^{\varphi(U)}$ multiplicity freely.
\end{claim}

 Assume towards a contradiction that $f, g \in \mathcal{O}(X)^{\varphi(U)}$ are linearly independent $\varphi(T)$-semi-invariants of the same $\varphi(T)$-weight.
  Hence, the  subgroups $f \cdot \varphi(U)$ and $g \cdot \varphi(U)$ of $\Aut(X)$ are root subgroups with respect to $\varphi(T)$ with the same weight. Then $\varphi^{-1}(f \cdot \varphi(U))$ and $\varphi^{-1}(g \cdot \varphi(U))$ are root subgroups with respect to $\varphi^{-1}(\varphi(T)) = T$.
Indeed, $\varphi^{-1}(f \cdot \varphi(U)) \setminus \{ \id_S \}$ is a $T$-orbit and hence is a quasi-affine curve in $\Aut(S)$. Therefore, $\varphi^{-1}(f \cdot \varphi(U)) =
(\varphi^{-1}(f \cdot \varphi(U)) \setminus \{ \id_S \}) \circ 
(\varphi^{-1}(f \cdot \varphi(U)) \setminus \{ \id_S \})
=\overline{\varphi^{-1}(f \cdot \varphi(U))}$ is an algebraic subgroup of $\Aut(S)$ that is normalized by algebraic torus $T$. Finally, by \cite[Lemma 4.10]{LRU18} $\varphi^{-1}(f \cdot \varphi(U)) \simeq \mathbb{G}_a$ and analogously 
 $\varphi^{-1}(g \cdot \varphi(U)) \simeq \mathbb{G}_a$.

  By Lemma \ref{hyperbolic}   root subgroups $\varphi^{-1}(f \cdot \varphi(U))$ and $\varphi^{-1}(g \cdot \varphi(U))$ with respect to $T$
  have different weights which means that $T$ acts on $\varphi^{-1}(f \cdot \varphi(U))$ and $\varphi^{-1}(g \cdot \varphi(U))$ with different kernels. This contradicts the assumption that $f \cdot \varphi(U)$ and $g \cdot \varphi(U)$ are $\varphi(T)$-root subgroups with the same weight. The claim follows.

	\medskip

Since $\varphi(T)$ acts on the invariant ring $\mathcal{O}(X)^{\varphi(U)}$ multiplicity freely, by Lemma \ref{Kr15local} $\dim X \le 2$. 
Moreover, we claim that $\dim X \neq 1$.
Indeed, if $\dim X = 1$, then since $X$ admits an action of $\varphi(U) \simeq \mathbb{G}_a$,
$X$ is isomorphic to the affine line $\mathbb{A}^1$.
Hence,   $\Aut(X=\mathbb{A}^1)$ is a two-dimensional algebraic group.
But this is not possible since
there is a non-trivial semi-invariant
$f \in \mathcal{O}(S)^U$ and then $f \cdot U$ is a root subgroup of $\Aut(S)$ with respect to $T$ that is different from $U$. Therefore, by Claim \ref{connectedtorus} $\varphi(U)$ and $\varphi(f \cdot U)$ are root subgroups of $\Aut(X)$ with respect to $\varphi(T)$, i.e. $\Aut(X)$ contains a three-dimensional algebraic subgroup  $\varphi(T) \ltimes (\varphi(U) \times \varphi(f\cdot U))$.
 This proves that  $\dim X = 2$. Finally, by \cite[Theorem 1.3]{LRU18} $X$ is non-toric and by
\cite[Theorem 1.2]{LRU18}  $X$ is a hyperbolic $\varphi(T)\simeq \mathbb{G}_m$-surface.

\section{Proof of Theorem \ref{main1}}

In the first paragraph of the proof of \cite[Proposition 1]{Sie96} it is shown that the Lie algebra of vector fields of a singular affine variety $X$ is non-simple. The same proof is suitable for the next statement. 
For the convenience of the reader we provide the detailed proof.

\begin{proposition}\label{non-simplicity}
Assume  $X$ is a singular affine variety and assume that  $\lielnd{X}$ is non-trivial. Then the Lie algebra $\lielnd{X}$ is not simple.
\end{proposition}
\begin{proof}
Denote    by $I \subset \mathcal{O}(X)$ the ideal corresponding to the singular locus of $X$. By \cite[Theorem 5]{Sei67} any vector field $\mu \in \Ve(X)$ preserves $I$, i.e. $\mu(I) \subset I$. Hence, $\mu(I^k) \subset I^k$ for any $k \in \mathbb{N}$.  Denote by 
\[
\lielnd{X,I^k} = \{ \nu \in \lielnd{X} \mid \nu(\mathcal{O}(X))\subset I^k \}.
\]
This is the ideal of the Lie algebra $\lielnd{X}$ since 
\[
[\nu,\mu](f) = 
\nu(\mu(f)) - \mu(\nu(f)) \in I^k,
\]
where $f \in \mathcal{O}(X)$, $\nu \in \lielnd{X,I^k}$ and $\mu \in \lielnd{X}$. 
It is clear that for a big enough $k$,
$\lielnd{X,I^k} \subset \lielnd{X}$ is proper.
Finally, $\lielnd{X,I^k}$ is non-zero. Indeed, assume 
 $\mu$ a locally nilpotent vector field on $X$ and $f \in I$. Hence, there exists $l \in \mathbb{N}$ such that $\mu^l(f) = 0$ and $g=\mu^{l-1}(f) \neq 0$. This means that $\mu(g) = 0$. Note that $g \in I$ by \cite[Theorem 5]{Sei67}.
Therefore, $g^k\mu$ is a locally nilpotent vector field and $g^k\mu(\mathcal{O}(X)) \subset I^k$.
\end{proof}

\begin{proof}[Proof of Theorem \ref{main1}]
Let  $\varphi\colon \Aut(\Dp) \to \Aut(X)$ be an isomorphism of ind-groups. 
 By Lemma \ref{vectorfields}
 $\varphi$ induces the  isomorphism  of Lie algebras \[ d\varphi\colon \lielnd{\Dp}  \rightarrow  \lielnd{X}. \] 
By \cite[Theorem 1]{LR20} $\lielnd{\Dp}$ is simple. 
Hence, from  Proposition \ref{non-simplicity} it follows that $X$ is smooth. Moreover, 
by Theorem \ref{main2}  $\dim X = 2$.  Now by  \cite[Theorem 1]{LRU19} $X$ is isomorphic to some $\Dq$ for some polynomial $q$ and by \cite[Theorem 3]{LR20} $X$ is isomorphic to $\Dp$.
\end{proof}

\begin{example}\label{exampleparabolic}
 Consider the product $\mathbb{A}^1 \times C$, where $C$ is a non-rational affine curve with the trivial automorphism group and no non-constant invertible  regular functions. 
 The surface $\mathbb{A}^1 \times C$ is a $\mathbb{G}_m$-surface of parabolic type as each point of curve $C$ is the fixed $\mathbb{G}_m$-point.
 Assume $Z$ is an affine variety with the trivial automorphism group that contains no rational curves and admits no non-constant invertible regular functions.
 We claim that
 $\Aut(\mathbb{A}^1 \times C)$ and $\Aut(\mathbb{A}^1 \times Z)$
  are isomorphic as ind-groups.

   Let
  $\varphi\colon \mathbb{A}^1 \times Z \rightarrow \mathbb{A}^1 \times Z$ be an
  automorphism of $\mathbb{A}^1 \times Z$. Assume $z \in Z$. The image $\varphi(\mathbb{A}^1 \times \{ z \})$ is a subvariety of $\mathbb{A}^1 \times Z$ isomorphic to $\mathbb{A}^1$. Since $Z$ does not contain rational curves, 
  $\varphi(\mathbb{A}^1 \times \{ z \}) = \mathbb{A}^1 \times \{ z' \}$ for some $z' \in Z$. Therefore, $\varphi$ induces an automorphism of $Z$ which is trivial by assumption.  
  So, $\varphi(x,z)=(\psi(x,z),z)$, where $\psi\colon \mathbb{A}^1 \times Z \to \mathbb{A}^1$ is a regular function. For each $z \in Z$, $\psi$ induces an isomorphism of $\mathbb{A}^1$. Hence, $\psi(x) = f(z)x + g(z)$, where $f,g \in \mathcal{O}(Z)$.  Moreover, since there are no  non-constant invertible
regular functions on $Z$,  $f \in \mathbb{C}^* = \mathbb{C} \setminus \{ 0 \}$. We conclude that 
  \[
  \Aut(\mathbb{A}^1 \times Z) = \{ 
  (x,z) \mapsto (fx + g(z),z) \mid f \in \mathbb{C}^*, g \in \mathcal{O}(Z)
  \} \simeq \mathbb{G}_m \ltimes \mathcal{O}(Z),
  \]
  and 
 analogously
 \[
  \Aut(\mathbb{A}^1 \times C) = \{ 
  (x,c) \mapsto (fx + g(c),c) \mid f \in \mathbb{C}^*, g \in \mathcal{O}(C)
  \} \simeq  \mathbb{G}_m \ltimes \mathcal{O}(C).
  \]
  These two groups are isomorphic as ind-groups since there is an isomorphism of ind-groups
  $\phi\colon \mathcal{O}(Z) \to \mathcal{O}(C)$.
  This
  proves the claim.
\end{example}





\end{document}